\documentclass[11pt]{amsart}
\usepackage{palatino}
\usepackage{amsmath,amssymb,amsthm,amscd, amsxtra, amsfonts,graphicx,fancybox,cancel,comment}

\usepackage[all]{xy}
\usepackage{amssymb}
\usepackage{latexsym}
\usepackage{amscd}
\usepackage{color}
\input amssym.def
\input amssym

\newtheorem{theorem}{Theorem}[section]
\newtheorem{lemma}[theorem]{Lemma}

\newtheorem{proposition}[theorem]{Proposition}
\newtheorem{corollary}[theorem]{Corollary}
\newtheorem{conjecture}[theorem]{Conjecture}
\newtheorem{remark}{Remark}

\def \<{\langle}
\def \>{\rangle}

\newcommand{\bea}{\begin{eqnarray}}
\newcommand{\eea}{\end{eqnarray}}
\newcommand{\be}{\begin {equation}}
\newcommand{\ee}{\end{equation}}
\newcommand{\g}{\frak g}

\newcommand{\Z}{\mathbb Z}

\newcommand{\C}{\mathbb C}

\usepackage{amssymb}

\begin{document}

\title[ ]{   On some vertex algebras related to  $V_{-1}(\frak{sl} (n) )$ and their characters }

\author{Dra\v{z}en Adamovi\'c and Antun Milas}
\address{Department of Mathematics, University of Zagreb, Croatia}
\email{adamovic@math.hr}

\address{Department of Mathematics and Statistics,
University at Albany (SUNY), Albany, NY 12222}
\email{amilas@math.albany.edu}

\maketitle

{\centering\footnotesize {\em Dedicated to Mirko Primc on the occasion of his 70th birthday} \par}

\begin{abstract} We consider several vertex operator (super)algebras closely related to $V_{-1}(\frak{sl} (n) )$, $n \ge 3$ : (a) the parafermionic subalgebra $K(\frak{sl}(n),-1)$ for which we completely describe  its inner structure, (b) the vacuum algebra $\Omega (V_{-1}(\frak{sl} (n) ) )$, and (c) an infinite extension $\mathcal U$ of $V_{-1}(\frak{sl} (n) )$ constructed by combining certain irreducible ordinary modules with integral weights. It turns out that $\mathcal U$ is isomorphic to  the coset vertex algebra $\frak{psl}(n|n) _1 / \frak{sl}(n)_1$, $n \ge 3$. We show that $V_{-1}(\frak{sl}(n))$  admits precisely $n$ ordinary irreducible modules, up to isomorphism. This leads to the conjecture that ${\mathcal U}$ is {\em quasi-lisse}. We present evidence in support of this conjecture: we prove that the (super)character of $\mathcal U$ is quasi-modular of weight one by virtue of being the constant term of a meromorphic Jacobi form of index zero. Explicit formulas and MLDE for characters and supercharacters are given for $\frak{g}=\frak{sl}(3)$ and outlined for general $n$.  We present a conjectural family of 2nd order MLDEs for characters of vertex algebras $\frak{psl}(n|n) _1$, $n \geq 2$.
We finish  with  a theorem  pertaining to characters
of $\frak{psl}(n|n)_1$ and $\mathcal U$-modules. 

 \end{abstract} 


 \section{introduction}
 
Orbifolding, coset constructions, and simple current extensions, are standard methods for producing new examples of vertex 
algebras. 
For irrational vertex algebras it is also important to consider {\em infinite} simple current extension. For instance, lattice vertex algebras are infinite simple current extensions
of the Heisenberg vertex algebras. Infinite simple current extensions are also  important in logarithmic conformal field theory. As demonstrated by the authors, the triplet vertex algebra $\mathcal{W}(p)$  (which is $C_2$-cofinite) is indeed an infinite simple current extension of the non $C_2$-cofinite singlet vertex algebra \cite{AdM-triplet}. 
 
In this paper the aim is to study the simple affine vertex operator algebra of level $-1$ for $\frak{sl}(n)$, denoted by $V_{-1}(\frak{sl}(n))$, and some of its subalgebras and infinite extensions. This vertex algebra is known to be irrational and non $C_2$-cofinite and has been studied from several points of view. Early work \cite{KW2001} was focused primarily on various properties of  characters of representations. The first author and Per\v{s}e obtained a complete classification of ordinary (atypical)  irreducible $V_{-1}(\frak{sl}(n) )$-modules
and fusion rules of ordinary modules \cite{AP} (see also \cite{AP0}). They also showed that $V_{-1}(\frak{sl}(n) )$ admits 
generic (or typical) series of irreducible representations. Kac and Wakimoto recently obtained a Weyl-Kac type character formula \cite{KW}
 involving higher rank partial theta series (cf. also \cite{BCR}). Asymptotic and modular-type properties of characters of $V_{-1}(\frak{sl}(n) )$-modules were studied recently in the work of Bringmann, Mahlburg and the second author
 \cite{BMM}. Although characters are mixed quantum modular forms \cite{BCR,BMM}, presently it seems difficult to formulate and prove a continuous version of the Verlinde formula of characters even for ordinary modules.
   
Instead of studying the vertex algebra $V_{-1}(\frak{sl} (n) )$, here we focus on two somewhat better behaved objects: (i) the parafermionic algebra(s) and (i) a certain infinite (simple current) 
 extension which we denote by $\mathcal U$. Both vertex algebras have interesting properties from an algebraic and number theoretic standpoints; we explore both aspects in great depth.
  
Let us outline the content and the main results. Throughout we assume that $n \geq 3$. We first review construction of the simple vertex algebra $V_{-1}(\frak{sl}(n) )$. Here we utilize the rank $n$ symplectic fermion vertex algebra $\mathcal{A}(n)$ \cite{Ab}, a certain lattice vertex algebra $V_L$, and the beta-gamma system (or the Weyl vertex algebra) $W_{(n)}$. Then $V_{-1}(\frak{gl}({n}))$ (and then of course $V_{-1}(\frak{sl}(n))$) embbeds inside the zero "charge" subalgebra $W_{(n)}^{(0)} \subset \mathcal{A}(n) \otimes V_L$. Similary, we obtain explicit realization of irreducible ordinary $V_{-1}(\frak{sl}(n))$-modules denoted by  $V_s$, $s \in \mathbb{Z}$ (see Proposition 2.3). 
  
  Then we move on study parafermionic and vacuum subalgebras. Recall that the parafermionic subalgebra $K(\frak{sl}(n),-1)$ is defined 
  as
  $$K(\frak{sl}(n),-1):=\{  v \in V_{-1} (\frak{sl}(n) ) : a(m) v=0, \ a \in M(1), m \geq 0 \},$$
  where $M(1)$ is the Heisenberg subalgebra,
and the vacuum algebra is similarly defined as 
$$\Omega_n:=\{  v \in V_{-1}(\frak{sl}(n)  ) : a(m) v=0, \ a \in M(1), m > 0 \}.$$
Our next result pertains to the structure of these vertex algebras.
  \begin{theorem} We have 
  \item[(1)]
 $$K(\frak{sl}(n), -1) \cong  \overline{M(1)}^{\otimes n}, $$
 where $\overline{M(1)}$ is the singlet vertex algebra of central charge $-2$ (cf. \cite{Ab, AdM-triplet,Wa}),
\item[(2)] and 
$$ \Omega_n \cong \mathcal A(n) ^{(0 )},$$
the charge zero subalgebra of the symplectic fermion vertex algebra. 
  \end{theorem}
  
  Then we consider an infinite extension of $V_{-1}(\frak{sl}(n) )$. We first prove that for every $n \geq 3$
  $$\mathcal U^{(n)}:=\bigoplus_{s \in \mathbb{Z}} V_{n s}$$
  has a simple vertex algebra structure for $n$ even, and $\mathbb{Z}$-graded vertex superalgebra if $n$ is odd.
  Then we can prove  
  \begin{theorem}
  \begin{itemize}
  \item[(1)] The vertex (super)algebra $\mathcal U:=\mathcal{U}^{(n)}$ has precisely $n$  ordinary irreducible modules $\mathcal U_i$,  $0 \leq i \leq n-1$, such that
  $$\mathcal U_i \cong \bigoplus_{s \equiv  i \ {\rm mod} \  n} V_{s}.$$
  \item [(2)] For $n \geq 3$, we have  
  $$\mathcal U \cong \frac{\frak{psl}(n,n)_1 }{ {\frak{sl}(n) }_1}.$$ 
  \end{itemize}
  \end{theorem}
Since our newly introduced vertex algebra has finitely many ordinary modules, it is natural to ask whether it is quasi-lisse in the sense of \cite{AK16}. As we are  currently unable to prove this property,  instead, we investigate the (super)characters 
  of $\mathcal U$ and of its modules.
 If a vertex algebra is quasi-lisse, then necessarily characters and supercharacters must be solutions of modular linear differential equation (MLDE) \cite{AK16}.
 In particular, solutions of such equations are known to be either modular (as in the case of ordinary admissible representations) or quasimodular 
 (as in the case of Deligne's series at non-admissible levels). 
 We prove  
 \begin{theorem} Characters and supercharacters of $\mathcal U$ are quasi-modular forms. More precisely, for $n$ even (resp. odd) the character 
 ${\rm ch}[\mathcal U](\tau)$  (resp. the supercharacter ${\rm sch}[\mathcal U](\tau)$) is a quasi-modular form (with a multiplier) of weight $1$ and depth $1$ on $\Gamma_0(n )$.
   \end{theorem}
Motivated again by \cite{AK16} we conjecture that the (super)character of $\mathcal U$ is a component of a vector-valued 
modular form coming from a modular linear differential equations (MLDEs). 
Compared to Deligne's series where this differential equation is of order two, here the situation is more complicated as the order of the equation grows with $n$. We hope to return to vector-valuedness and properties of MLDE in our future publications. Here we only analyze an MLDE corresponding to  $\frak{g}=\frak{sl}(3)$ (see Proposition \ref{MLDE}).
 
The vertex algebras associated to  $\frak{psl}(n|n)$ and  $\frak{gl}(n|n)$  have attracted much attention in the literature (cf.  \cite{Ad-2017},  \cite{AKMPP-JA},  \cite{AKMPP-2018}, \cite{CG-2017},  \cite{CG-2018}). In the present paper we identify the  coset $\frak{psl}(n|n) _1 / {\frak{sl}(n) }_1$  as a vertex algebra $\mathcal U$ for $n \ge 3$. We   prove in Theorem \ref{psl3} that the super character for the simple vertex algebra $V_1(\frak{psl}(n,n))$ is for every $n \ge 2$ equal to the supercharcter of symplectic fermion vertex algebra, and equals to $\eta(\tau) ^2$.    In the case $n=3$ we present a different proof, by using we use the  (super)character of $\mathcal U$ from previous section, together with a branching rules for conformal embeddings in the case $n=3$.

  We have the following conjecture based on the analysis in the case $\frak{psl}(n|n)_1$ and results  from the paper \cite{AKMPP-JA} and \cite{AK16}.

\begin{conjecture}
For every $n \ge 0$ even  we have
$${\rm sch}[V_{-2} (\frak{osp}(n+8 \vert n) ) ](\tau)   =   {\rm ch} [V_{-2} (\frak{so}(8)) ](\tau). $$
\end{conjecture}
 
We should also mention that the vertex algebra $V_{-2} (\frak{osp}(n+8 \vert n) )$ has recently appeared in the work of K. Costello and D. Gaiotto \cite[Section 5]{CG-2018} in the context of $SU(2)$-gauge theory with $N \ge 4$ flavors.

{\bf Acknowledgments:} 
We would like to thank T. Creutzig and M. Gorelik   for valuable  discussions.

D.A. is  partially supported by the Croatian Science Foundation under the project 2634 and by the
QuantiXLie Centre of Excellence, a project cofinanced
by the Croatian Government and European Union
through the European Regional Development Fund - the
Competitiveness and Cohesion Operational Programme
(KK.01.1.1.01). A.M. was partially supported by the NSF Grant DMS-1601070.

   \section{The affine  vertex algebra $V_{-1} (\frak{sl}(n)  )$ } 
   In this section we recall the basic properties of  the affine vertex algebra $V_{-1} (\frak{sl}(n)  )$. Here we use the standard notation: $V^k(\frak{g})$ denotes the universal vertex 
   algebra of level $k$ and $V_k(\frak{g})$ is the corresponding simple vertex algebra. All affine vertex algebras are equipped with the usual conformal structure (via Sugawara's construction).
      \label{realization-sln}
   \subsection{Symplectic fermions and  the $c=-2$ singlet vertex algebra}
   The symplectic fermion vertex algebra  $\mathcal A (n) $  (see \cite{Ab} for more details) is the universal vertex superalgebra  generated by odd fields/vectors  $b_i$ and $c_i$ $(i =1 , \dots,n)$  with the following non-trivial  $\lambda$--bracket 
$$ [ (b_i)  _{\lambda} c_j ] = \delta_{i,j} \lambda.  $$
 $\mathcal A (n)$ can be realized  on the  irreducible  level one module for the Lie superalgebra with generators
$$\{ K, b_i(n), c_i(n), n \in {\Z} \}  $$ and  relations
\bea  
 && \{ b_i (n), b_j (m) \}  = \{ c_i (n), c_j (m)\}  = 0, \quad  \{b_i  (n), c_j (m) \} = n \delta_{i,j} \delta_{n+m,0} K.   \nonumber  
\eea
Here $K$ is central and other super-commutators are trivial. As a vector space, 
$$\mathcal A(n)  = \bigwedge {\rm span} \left\{ b_i(-m), c_i(-m), \ m \in {\Z}_{>0}, i=1, \dots, n  \right\}. $$
The fields  $b_i$,  $c_i$  can be identified as formal Laurent series  acting on $\mathcal A(n)$.
$$b_i(x) = \sum_{n \in {\Z}} b_i (n) x^{-n-1}, \quad  c_i(x) = \sum_{n \in {\Z}} c_i (n) x^{-n-1} $$

The vertex algebra $\mathcal A (n) $  has the following Virasoro element of central charge $c=-2 n$:
$$\omega_{\mathcal A(n)}  = \sum_{i=1} ^n :b_i c_i :. $$
There is a charge operator $J\in \mbox{End} (\mathcal A (n) )$ such that
$$[J, b_i (n)] = b_i(n), \quad [J, c_i (n)] = - c_i(n) $$
 which defines on $\mathcal A(n)$ the   $\Z$--gradation: 
$$\mathcal A(n) = \sum_{\ell \in {\Z}} \mathcal A(n) ^{(\ell)},   \quad  \mathcal A(n) ^{(\ell) } = \{ v \in \mathcal A(n) \ \vert \ J v = \ell v \}. $$

The vertex  algebra $\mathcal A(1) ^{(0)}$ is isomorphic to the singlet vertex algebra $\overline{M(1)}$  of central charge $c=-2$ (cf. \cite{Wa}, \cite{Ad-2003}).
For every $i \in \{0, \dots, n\}$  we set $\mathcal G _i ^0 =  1 $, and for $ m \in {\Z_{\ge 1}}$ we define
$$ {\mathcal G}_i ^m = b_i(-m)  \cdots b_i(-1) , \quad  {\mathcal G}_i ^{-m} = c_i(-m)  \cdots c_i(-1) . $$
Each  $ u_r = {\mathcal G}_i ^r . {\bf 1} $, $r \in {\Z}$,  is a singular vector for the singlet vertex algebra, which generates an irreducible module $\pi_r$ (note that we drop the index $i$).

It was proven in \cite{AdM-2017} that modules are simple current $\mathcal A(1) ^{(0)}$--modules with the following fusion rules:
$$ \pi_{r} \times \pi_s = \pi_{r+s}. $$

\subsection{The Clifford vertex algebra}
 The Weyl vertex algebra $F_{(n)}$ is the universal vertex algebra  generated by the   odd fields ${\Psi}_i ^{\pm} $ and the following  non-trivial $\lambda$--bracket:
$$ [ (\Psi_i ^+ )_{\lambda} \Psi_j ^-] = \delta_{i,j}, \quad (i,j = 1\dots, n). $$
  The vertex algebra $F_{(n)}$    has the structure  of the irreducible  level one  module for Clifford algebra  with generators
$ \{ K, \Psi^{\pm} (n+ 1/2) \ \vert  \ n \in {\Z} \}$ and super-commutation relations:
$$\{ \Psi_i^+ (r), \Psi_j ^- (s) \}  =\delta_{i,j} \delta_{r+s,0} K , \quad  \{\Psi_i^{\pm}  (r), \Psi_j ^{\pm}  (s) \} =0 \quad (r,s \in \tfrac{1}{2} + \Z, \ i, j = 1, \dots, n), $$
where $K$ is central element. The fields $\Psi^{\pm} _i $  acts on $F_{(n)} $ as the following Laurent series
$$\Psi  ^{\pm} (z)  = \sum_{ n \in {\Z} } \Psi ^{\pm} (n+\frac{1}{2})  z^{-n-1}.$$

\subsection{The Weyl vertex algebra and  its bosonization.} 
The Weyl vertex algebra $W_{(n)}$ is the universal vertex algebra  generated by the  even fields ${a}_i ^{\pm} $ and the following  non-trivial $\lambda$--bracket:
$$ [ (a_i ^+ )_{\lambda} a_j ^-] = \delta_{i,j}, \quad (i,j = 1\dots, n). $$
  The vertex algebra $W_{(n)}$    has the structure  of the irreducible  level one  module for the Lie  algebra  with generators
$ \{ K, a^{\pm} (n+ 1/2) \ \vert  \ n \in {\Z} \}$ and commutation relations:
$$[a_i^+ (r), a_j ^- (s) ] =\delta_{i,j} \delta_{r+s,0} K , \quad  [a_i ^{\pm}(r) , a_j^{\pm}(s)]=0 \quad (r,s \in \tfrac{1}{2} + \Z, \ i, j = 1, \dots, n), $$
where $K$ is central element. The fields $a^{\pm} _i $  acts on $W_{(n)} $ as the following Laurent series
$$a ^{\pm} (z)  = \sum_{ n \in {\Z} } a^{\pm} (n+\frac{1}{2})  z^{-n-1}.$$
For $i= \{1, \dots, n \}$ and $r \in {\Z}$ , we define $$X_i ^r :=  a_i ^+ (-1/2 ) ^r, \ \ r \ge 0 \ \  {\rm and} \  \ X_i ^r := a_i ^-  (-1/2)  ^{-r} {\bf 1}, \ r < 0.$$
Let $V_L = M_n (1)  \otimes {\C}[L] $ be the lattice vertex superalgebra associated to the lattice $$L = {\Z}  \varphi_1 \oplus \cdots \oplus  {\Z}  \varphi_n$$  with products:
$$\langle \varphi_i, \varphi_j \rangle = - \delta_{i,j} \ \   (i,j = 1, \dots , n). $$
 Here $M_n (1)$ denotes the level one module for the Heisenberg vertex algebra associated to the Heisenberg Lie algebra 
 $\widehat {\frak h}_n   = {\C}[t,t^{-1}] \otimes {\frak h}_n \oplus {\C} K$, where  ${\frak h}_n =  {\C} \otimes _{\Z} L$.
  
 \subsection{Realization of $V_{-1} (\frak{sl}(n))$ and its ordinary modules}
 
 
 We have the embedding $$ W_{(n)} \rightarrow  \mathcal A(n)  \otimes V_ L $$ such that
 $$ a_i ^+ =  : b_i e^{\varphi _i } :  \quad , \quad   a_i ^{-} = - : c_i e^{-\varphi_i} : $$
 Define $c: = - (\varphi _1 + \cdots \varphi_n)$. Then $c(0)$ defines on $W_{(n)}$ the  natural ${\Z}$--gradation:
 $$ W_{(n)} = \bigoplus _{ s  \in {\Z} }  W_{(n)} ^{(\ell ) }. $$
 Let $M_c(1)$ be the Heisenberg vertex algebra of level $1$ generated by $c$. Let $M_{c} (1, r)$ be the irreducible $M_c(1)$--module  on which  $c(0)$ acts as  $r \mbox{Id}$. 
 
 The  vertex subalgebra of  $W_{(n)}^{(0)}$ generated by the vectors 
$$\{ e_{i,j} = - :a _i ^+  a _j  ^- :  \ \vert \ i,j = 1, \dots, n \} $$ is isomorphic  to the simple affine vertex algebra
$V_{-1} ({\frak gl} (n) ) $ at level $-1$ (cf. \cite{AP}).

We also have for $i \ne j$:
\bea e_{i,j} :=  : b _i  c_j e^{\varphi_i - \varphi_j} :. \label{lattice-formula} \eea
Then we have
  \begin{theorem} 
 $ W_{(n)} ^{( 0 ) } $ is a simple vertex algebra and the following holds:
 \begin{itemize}
 \item \cite{KR} For $n=1$,  $W_{(n)} ^{( 0 ) }$  is  a $\mathcal W_{1+\infty}$-algebra  at central charge $c=-1$.
  
 \item   \cite[Theorem 5.2] {CKLR} For $n =2$, $W_{(n)} ^{( 0 ) } \cong   \mathcal W \otimes M_{c} (1) $, where $\mathcal W$ is  a certain  $W$-algebra of type $W(1,1,1,2,2,2)$ at central charge $c=-3$ (conjecturally isomorphic to $W_{-5/2}(\frak{sl} (4), f_{sh})$ (cf.   \cite{CKLR}, \cite{AMP}) where $f_{sh}$ is a short nilpotent element of $\frak{sl} (4)$  
 
 
  \item  \cite{AP} For $n \ge 3$: $W_{(n)} ^{( 0 ) }  \cong V_{-1} ({\mathfrak sl}(n)) \otimes M_c(1)$.
  
  \end{itemize}
  
 \end{theorem}
 
     We need the following result on fusion rules.
     \begin{proposition}  \cite{AP}
     Assume that $n \ge 3$.
  For $s \in \Z_{\ge 0}$, let
  $$V _s := L_{\frak{sl} (n)} (-(1+s)  \Lambda _0 + s \Lambda_1), \quad V_{-s} := L_{\frak{sl} (n)} (-(1+s)  \Lambda _0 + s \Lambda_{n-1}). $$
  \begin{itemize}
  \item The set $\{ V _s \ \vert \ s \in {\Z} \}$ provides a complete list of irreducible $V_{-1} (\frak{sl} (n))$ modules in the category $KL_{-1}$ (= the category  of ordinary modules).
  \item The following fusion rules hold in the category $KL_{-1}$.
 \bea \label{fusion-rules} V_{s_1} \times V _{s_2} =  V _{s_1 + s_2} \qquad (s_1, s_2 \in \Z). \eea
  \end{itemize}
  \end{proposition}

Let us now present a realization of irreducible $V_{-1}(\frak{sl} (n))$--modules. Let  
$$ Q_n = \{ z_1 \varphi_1 + \cdots + z_n \varphi_n  \ \vert \ z_1 + \cdots + z_n = 0 \}. $$ 
Since $V_{-1} (\frak{sl} (n)) \subset \mathcal A(n)  \otimes V_{Q_n}$, we have that for every $\lambda \in Q_n ^0$ (= the dual lattice of $Q_n$), $\mathcal A(n) \otimes V_{\lambda + Q_n}$ is a $V_{-1} (\frak{sl}(n))$--module.
Let $$\omega_1 = \frac{1}{n} ((n-1)  \varphi_1 - \varphi_2 - \cdots -  \varphi_n), \quad \omega_{n-1}= \frac{1}{n} ( \varphi_1 +  \cdots + \varphi_{n-1}  - (n-1)  \varphi_n). $$

We set $v^{(0)} ={\bf 1}$.  For $i  \in {\Z}_{> 0}$  we define
   \bea  v^{ (j)} & = &  b_1(-  j  ) \cdots  b_1 (-1)  {\bf 1} \otimes e^{  j    \omega_1}  \nonumber \\
 v^{  (- j )} & = &  c_n(- j )  \cdots c_n (-1) {\bf 1}  \otimes e^{  j  \omega_{n-1}}  \nonumber \eea

\begin{proposition} \label{real-mod} For $s \in {\Z}$ we have:
$$V _s  \cong V_{-1} (\frak{sl} (n)). v^{(s)}. $$
\end{proposition}
 \begin{proof}
 First we notice that $v^{(s)}$ is a singular vector for $\widehat{sl(n)}$. Then   $\widetilde U _s  = V_{-1} (\frak{sl}(n)). v^{(s)} $ is a highest weight $V_{-1} (\frak{sl} (n))$--module, having the same highest weight as $U_s$.
 By using the bosonization of the Weyl vertex algebra, we show that  as a $V_{-1}(\frak{gl}(n)) = V_{-1} (\frak{sl} (n)) \otimes M_c(1)$--module
 $ W_{(n)} ^{(s)} \cong \widetilde U _s \otimes M_c(1, s)$. 
 Since $W_{(n)} ^{(s)}$ is irreducible $V_{-1}(\frak{gl} (n))$--module, we conclude that $\widetilde U _s$ is irreducible  $V_{-1} (\frak{sl} (n))$--module, and thus $\widetilde U _s \cong U_s$.
 \end{proof}

\section{ Parafermionic algebra and the vacuum of $V_{-1}(\frak{sl} (n))$ }

Recall the definition of the parafermion vertex algebra of level $k$:
 $$ K(\g, k) := \{ v \in V_k(\g) \  \vert \   ({\frak h} \otimes t ^{m} ). v = 0, \ m \in {\Z}_{\ge 0} \}.$$
 \begin{theorem} Assume that $ n \ge 3$. Then 
 $$K(\frak{sl}(n), -1) \cong (  \mathcal A(1) ^ 0 )^{\otimes n}. $$
 \end{theorem}
 \begin{proof}
 Let $M_{n-1} (1)$ (resp. $M_n(1)$ )  be the Heisenberg vertex algebra generated by the Cartan Lie subalgebra of $\frak{sl} (n)$  (resp. $\frak{gl} (n)$).  Let $\frak{gl} (n) = \frak{sl} (n) \oplus {\C} c$. As usual we identify
 $x = x_{(-1)} {\bf 1}$ for $ x\in \frak{sl} (n)$. Then $M_{n} (1) = M_{n-1} (1) \otimes M_c(1)$, where $M_c(1)$ is the  Heisenberg  vertex algebra generated by $c$.
 
 By \cite{AP}, we have for $n \ge 3$:
 $$ V_{-1} (\frak{gl}(n)) = V_{-1} (\frak{sl}(n)) \otimes M_c(1) \cong  \left( W_{(n)} \right)^0 = \mbox{Ker}_{   W^{(n) } } c (0). $$
 A. Linshaw in \cite[Theorem 7.2]{L-JPAA} proved that 
 $$\mbox{Com} (M_n(1), W^{(n)} ) \cong  (  \mathcal A(1) ^ 0 )^{\otimes n}. $$
 (This corresponds to the case $m=n$ in \cite{L-JPAA}).
Therefore for $n \ge 3$:
$$ \mbox{Com} (M_n(1), W_{(n)} ) \cong \mbox{Com} (M_{n}(1), V_{-1} (\frak{gl}(n))  ) \cong  K(\frak{sl}(n),-1).  $$
The proof follows.
 \end{proof}

\subsection{The vacuum space}

The vacuum space is defined as
$$\Omega (V_{k} (\g) ) = \{ v \in V_{k} (\g) \ \vert \ h(j) v = 0 \quad j\ge 1, \ h \in {\mathfrak h} \},$$
and it has the structure of a generalized vertex algebra \cite{Li-abelian}, \cite{DL}.  \begin{theorem}
\item[(1)] Assume that $n \ge 2$.  The vacuum algebra $$ \Omega_n := \Omega (V_{-1} (sl(n)) = \{ v \in V_{-1} (sl(n)) \ \vert \ h(j) v = 0 \quad j\ge 1, \ h \in {\mathfrak h} \}$$
is  isomorphic  to a vertex subalgebra  of  $ \mathcal A(n) $ generated by 
$$ \{  Z_{i,j} = : b_i c_j : \ \vert \ 1 \le i \ne j \le n \}. $$

\item[(2)]  Assume that $n \ge 3$. Then $ \Omega_n \cong \mathcal A(n) ^{(0 )} $ 

\item[(3)]   The $q$--character of $ \Omega_n$ is given by
$${\rm ch}[\mathcal A(n) ^{(0 )} ](\tau)= q^{\frac{1}{12}} \mbox{\rm CT}_\zeta \prod _{i = 1} ^{\infty}   ( 1 + q^{i} \zeta) ^n  ( 1 + q^{i} \zeta^{-1}) ^n ,$$
where ${\rm CT}$ is the constant term.
\end{theorem}
 \begin{proof}  
 
 %

 The proof uses 
  the explicit realization, the bosonization and the formula for $Z$--operators.
  
  By using \cite[Theorem 6.4]{Li-abelian}, we see that  $\Omega (V_{-1} (sl(n))$ is generated by the following (generalized) vertex operators
$$ Z_{i,j} (z) = Y_{\Omega} (e_{i,j}, z):= E^-( - h_{i,j}, z) e_{i,j} (z)  (z) E^+ ( - h_{i,j}, z) z^{h_{i,j} (0)} $$
where
$ h_{i,j} = \varphi_i - \varphi_j$ and 
$$ E^{\pm} (\alpha, z) = \exp  \left(   \sum_{n =1} ^{\infty}  \frac{\alpha(\pm n)}{\pm n} z^{\mp n}\right).$$
Using (\ref{lattice-formula}) we see that on  $\Omega_n=\Omega (V_{-1} (sl(n))$ we have that
$$ Z_{i,j}  = :b_i c_j: . $$

 Therefore $\Omega_n$ is generated by
$ \{  Z_{i,j} = : b_i c_j : \ \vert \ 1 \le i \ne j \le n \}. $ This proves (1).
 
 (2) 
 Since $\Omega_n$ is generated by $Z_{i,j}$, $i \ne j$ we have:
 $$u_{i,j} =  (  Z_{i,j}  )_1 Z_{j,i} =  (b_i (-1) c_j (-1) {\bf 1}  )_1 b_j (-1) c_i(-1) {\bf 1} = : b_i c_i : + :b_j c_j: \in \Omega_n. $$
 This implies  
  $$  (*)  : b_i c_ i : \in  \Omega_n$$
 Since $\Omega_n \subset \mathcal A(n) ^0$, then  (2) will   follow from the following  claim
 \begin{itemize}
 
\item [(2')] For $n \ge 3$, $\mathcal A(n) ^0$  is generated by  the set $\{ : b_i c_j: , \quad i, j = 1, \dots, n \}$.
\end{itemize}
 
 The claim (2') can be proved using completely analogue   methods to those from \cite[Section 5]{CKLR} (we omit details).

The proof of assertion (3) is clear.

 \end{proof}

\begin{remark}
In \cite{CKLR}, the authors denoted the maximal vertex operator subalgebra of the generalized vertex operator  algebra $\Omega (V_k(\g))$ by $E_{k,\g}$ (see \cite[Example 3]{CKLR}).  In our case, $\Omega (V_{-1}(sl(n)))$ is a vertex algebra, so we have
$$ E_{-1, sl(n)} \cong     \mathcal A(n) ^{(0 )}.$$
\end{remark}

 One can consider the $V_{-1} (sl(n) )$--module
 $$ \mathcal U^{large} = \bigoplus_{s \in {\Z} } V_s,$$
 and show that it is a generalized vertex algebra.
 
 On the other hand, one can prove the following theorem.
 
 \begin{theorem} \label{extended-algebra}
 \item[(1)] The $V_{-1} (\frak{sl}(n)) )$--module
 $$ \mathcal U^{ (n)}  =  \bigoplus_{s \in {\Z} } V_{n s},$$
 carries the structure of a vertex operator algebra if  $n$ is even and a  ${\Z}$--graded vertex operator superalgebra if $n$ is odd.
 \item[(2)] In the category of ordinary modules,  $\mathcal U ^{(n)} $ has $n$--non-equivelent irreducible ordinary  modules:
 $$\mathcal U_i :=  \bigoplus_{s \in {\Z} } V_{n  s+ i} \quad (i= 0, \dots, ,n-1)$$
 with the following fusion rules $$\mathcal U_i \times \mathcal U_j = \mathcal U_{i+j \mod  n}.$$
 \end{theorem}
 \begin{proof}
 The proof of assertion  (1) is based on the explicit realization  discussed in Section  \ref{realization-sln}.
 %
 Note that $ m  n \omega_1, m n  \omega_{n-1}  \in L$ and that $e^{m n\omega_1}$ and $ e^{m n  \omega_{n-1}}$ are even vectors in $V_L$.
 
 Consider   the vertex subalgebra $\widetilde {\mathcal U}$  of   $\mathcal A(n) \otimes V_L$ generated by $U_0$ and highest weight vectors
 \bea  v^{ (n)} & = &  b_1(-n ) \cdots  b_1 (-2) b_1 (-1) \otimes e^{  n  \omega_1}  \nonumber \\
 v^{ (-n)} & = &  c_n(-n) \cdots c_n  (-2) c_n  (-1) \otimes e^{  n \omega_{n-1}} \nonumber \eea
Note that vector $v^{(\pm n ) }$ have conformal  weight $n$.  Moreover, vectors  $v^{(\pm n ) }$ are even (resp. odd) if $n$ is even (resp. odd). Therefore, $\widetilde{\mathcal U}$ is  a vertex operator algebra if $n$ is even, and a $\Z$--graded vertex operator superalgebra if $n$ is odd.

By  Proposition \ref{real-mod} we have that $m \in {\Z}_{> 0}$ modules $V_{ \pm m n }$ are realized as  $ V_{\pm  m n } = V_0 . v^{(\pm  m n )}$.
 Since $v^{(\pm  m n )} \in \widetilde {\mathcal U}$  we get that 
 $ \widetilde {\mathcal U} \supset \mathcal U^{(n)}. $     By using fusion rules  (\ref{fusion-rules}) we  see that   $\mathcal U^{(n)}$ is a vertex subalgebra of  $ \widetilde {\mathcal U}$. Since both vertex algebras are generated by  $U_0$ and $v^{(\pm   n )}$ we conclude that    $\mathcal U^{(n)} =  \widetilde {\mathcal U}$. 
 This proves the assertion (1).
 
 Let us now discuss the construction and classification of irreducible $\mathcal U^{(n)}$--modules.
 Clearly $\mathcal L_{i} = \mathcal U^{ (n)} v^{(i)} = \bigoplus_{s \in {\Z} } V_{n  s+ i} $ is an irreducible $\mathcal U ^{(n)}$--modules for $i = 0, \dots, n-1$.

 Assume that $M$ is an irreducible ordinary module for $\mathcal U^{(n)}$. Then $M$ is in the category $KL_{-1}$ as  a $V_{-1}(sl(n))$--module. Since the top component $\Omega (M)$ is a finite-dimensional module for $U(sl(n))$, we conclude that $\Omega(M)$ contains a singular vector for $\widehat{sl(n)}$. Thus, $M$ contains a $V_{-1}(sl(n))$--submodule isomorphic to  $V_i$ for certain  $i \in {\Z}$. By using the   fusion rules  (\ref{fusion-rules})  again, we conclude that 
 $ M \cong \bigoplus_{s \in {\Z} } V_{n  s+ i}  = \mathcal U_i $. 
 The proof follows.
 \end{proof}
 
 \begin{conjecture} \label{quasi-lise}
 The vertex algebra $ \mathcal U^{ (n)}$ is quasi-lisse in the sense of \cite{AK16}.
 \end{conjecture}
 
 \begin{remark}
In our paper we present some evidence for Conjecture \ref{quasi-lise}.
 \begin{itemize}
 \item  There are finitely many (ordinary) irreducible  $ \mathcal U^{ (n)}$--modules.
 \item Characters and super-characters of (ordinary)  $ \mathcal U^{ (n)}$--modules are  quasi-modular forms. For $\frak{g}=sl_3$, the supercharacters are solutions of an MLDE (see Proposition \ref{MLDE}). 
 \item The vacuum spaces is a $C_2$--cofinite vertex operator algebra.
 \item[] For simplicity, let us discuss the case $n=3$. Then we will see that the vacuum space $\Omega (  \mathcal U^{ (3)} ) $ is a ${\Z}_3$--orbifold of the symplectic fermion vertex algebra $\mathcal A(3)$.
 Since every cyclic orbifold of a $C_2$--cofinite vertex algebra is $C_2$--cofinite (cf. \cite{Miy}), then the vacuum $\Omega (  \mathcal U^{ (3)} ) $ is $C_2$--cofinite. 
 \end{itemize}
 
 \end{remark}

 \begin{remark}  Assume that $V$ is a  vertex operator (super) algebra containing a Heisenberg vertex subalgebra $M(1)$.
We believe that if the vacuum space $\Omega(V)$ is $C_2$--cofinite, then $V$ is  quasi-lisse.
 \end{remark}
 
 

   \section{ A decomposition of $V_{-1} (\frak{sl}(n))$ as $K(\frak{sl}(n),-1) \otimes M_{n-1} (1)$--module }
   
 \subsection{ A decomposition of $V_{-1} (\frak{sl}(3))$ as $K(\frak{sl}(3),-1) \otimes M_2 (1)$--modules : from the realization}

Let $Q$ be the root lattice of $\frak{sl}(3)$. For $(r,s) \in {\Z} ^2$, we set 
$\gamma_{r,s} =   r  \varphi_1  + s  \varphi_2  - (r+s)    \varphi_3 , $
We have:
\bea
V_{-1} (\frak{sl}(3)) & = & \bigoplus_{ (r,s) \in {\Z} ^2 }   \left( K(\frak{sl}(3), -1) \otimes M_2 (1)  \right) . P_{r,s}  \nonumber \\
& = & \bigoplus_{ (r,s) \in {\Z} ^2 }   \left( K(\frak{sl}(3), -1) \otimes M_2 (1)  \right) . ( v_{r,s} \otimes e^ { \gamma_{r,s} })  \nonumber \\
&=&  \bigoplus_{ (r,s) \in {\Z} ^2 }       K_{r,s} \otimes  M_2 (1). e^ { \gamma_{r,s} } \nonumber 
%
\eea
where
$$  P_{r,s} =    X_1 ^ r X_2 ^s X_3 ^{-r-s}  {\bf 1}=  v_{r,s} \otimes e^ { \gamma_{r,s} }  ,$$
and
 $$ K_{r,s}  = \{ v \in V_{-1} (\frak{sl}(3)) \ \vert \   h (n) v = \delta_{n,0} \langle h, \gamma_{r,s} \rangle  v \ \ \forall h \in {\mathfrak  h } , \ n \in {\Z}_{\ge 0}  \}$$
is irreducible $K(\frak{sl}(3), -1)$--module generated by  lowest weight vector 
 $$v_{r,s} = \mathcal G _1^r  \mathcal G _2 ^s \mathcal G_r ^{-r-s} {\bf 1}. $$ 
We conclude that $ K_{r,s} = \pi_r \otimes \pi_s \otimes \pi _{-r-s}$. In this way we have proved the following theorem:

\begin{theorem}  \label{decomposition-sl(3)} The vertex algebra $V_{-1} (\frak{sl}(3))$ is a simple current extension of $( \mathcal A(1) ^0 ) ^{\otimes 3}  \otimes M_2 (1)$  and
$$ V_{-1} (\frak{sl}(3)) =  \bigoplus_{ (r,s) \in {\Z} ^2 }    
\pi_r \otimes \pi_s \otimes \pi _{-r-s} \otimes M_2 (1)  e^ { \gamma_{r,s} }  $$ 
\end{theorem}
  
 \subsection{ A decomposition of $V_{-1} (\frak{sl}(3))$ as $K(\frak{sl}(3),-1) \otimes M_2 (1)$--modules II: from character formulas}
 
 It is possible to prove Theorem \ref{decomposition-sl(3)} directly from the character (\ref{charUs}).  Observe a well-known identity  \cite{An}
$$\frac{1}{\prod_{n \geq 1} (1-zq^{n-1})(1-z^{-1} q^n)}=\frac{\sum_{m \in \mathbb{Z}} F_m(q) z^m}{\prod_{n \geq 1} (1-q^n)^2},$$
where 
$$F_m(q)=\sum_{r \geq 0} q^{(2r+1)r+2mr}-\sum_{r \geq 0} q^{(2r-1)r+m(2r-1)}, \ \ m \geq 0$$
and
$$F_m(q)=q^m F_{-m}(q), m <0 $$
We use this formula to expand the character (three times). Then we extract the coefficients of $z_1$ and $z_2$ which corresponds the modules for the vacuum algebra. Finally 
we use a well-known formula for ${\rm ch}[K_{r,s}]$  \cite{BM2015}  irreducible module for the tensor product of three copies of the singlet algebra. This gives 
$${\rm ch}[V_{-1}(\frak{sl}_3)](\tau)=\sum_{ (r,s) \in {\Z} ^2 }   {\rm ch}[ K_{r,s}](\tau) {\rm ch}[M_2 (1). e^ { \gamma_{r,s} }](\tau). $$

\subsection{$q$-hypergometric formula for ${\rm ch}[V_{-1}(\frak{sl}(3))](\tau)$}
Now we use discussion from the last section to prove
\begin{proposition}
\begin{eqnarray*}
& (q;q)_\infty^2 q^{-\frac16} {\rm ch}[V_{-1}(\frak{sl}_3)](\tau)  \\
& =\sum_{k_1,k_2 \in \mathbb{Z}^2} \sum_{n_1,n_2,n_3 \geq 0} \frac{q^{n_1^2+n_2^2+n_3^2+(|k_1|+1)n_1+(|k_2|+1)n_2+(|-k_1-k_2|+1)n_3+\frac{|k_1|+|k_2|+|-k_1-k_2|}{2}}}{(q;q)_{n_1}(q;q)_{n_1+|k_1|}(q;q)_{n_2} (q;q)_{n_2+|k_2|}(q;q)_{n_3}(q;q)_{n_3+|-k_1-k_2|}}
\end{eqnarray*}
\end{proposition}
\begin{proof}
Follows directly from the $q$-hypergeometric representations of the $p=2$ false theta functions which are known to be 
characters of modules for the $p=2$ singlet algebra (here $k \in \mathbb{Z}$):
$$\frac{q^{\frac{k}{2}} \sum_{n \geq 0} \left( q^{2n^2+n(2k+1)}-q^{2n^2+n(2k+3)+k+1} \right)}{(q;q)_\infty}=\sum_{n \geq 0} \frac{q^{n^2+n(|k|+1)+\frac{|k|}{2} }}{(q;q)_n(q;q)_{n+|k|}}.$$
These relations are implicitly proven in \cite{BM2015, War} generalizing a well-known Ramanujan's identity corresponding to $k=0$.
\end{proof}


 \subsection{ A realization of the vertex superalgebra $\mathcal U ^{(n)} $}
  
 Let $$\omega_1 = \frac{1}{n} ((n-1)  \varphi_1 - \varphi_2 - \cdots -  \varphi_n), \quad \omega_{n-1}= \frac{1}{n} ( \varphi_1 +  \cdots + \varphi_{n-1}  - n \varphi_n). $$
 
 Note that $ m  n \omega_1, m n  \omega_{n-1}  \in L$ and that $e^{m n\omega_1}$ and $ e^{m n  \omega_{n-1}}$ are even vectors in $V_L$.
 
 The vertex superalgebra $\mathcal U^{(n)} $ can be realized as a subalgebra of $\mathcal A(n) \otimes V_L$ generated by $U_0$ and highest weight vectors
 \bea  v^{ (n)} & = &  b_1(-n ) \cdots  b_1 (-2) b_1 (-1) \otimes e^{  n  \omega_1}  \nonumber \\
 v^{ (-n)} & = &  c_n(-n) \cdots c_n  (-2) c_n  (-1) \otimes e^{  n \omega_{n-1}} \nonumber \eea
 Note that vector $v^{(\pm n ) }$ have conformal  weight $n$. Therefore, $\mathcal U$ is $\Z$--graded vertex operator superalgebra.

  \subsection{The vacuum space $\Omega (\mathcal U ^{(3)}) $} 
  Let us again consider the case $n=3$, so that $\mathcal U = \mathcal U ^{(3)}$.

  Moreover,  for $m \in {\Z}_{> 0}$ modules $U_{ \pm 3 m}$ are realized as  $ U_{\pm 3 m} = U_0 . v^{(\pm 3m )}$ where
 
   \bea  v^{ (3 m)} & = &  b_1(-  3 m ) \cdots  b_1 (-1) \otimes e^{  3 m   \omega_1}  \nonumber \\
 v^{ (-3m )} & = &  c_3(-3 m )  \cdots c_3 (-1) \otimes e^{  3 m  \omega_2}  \nonumber \eea


\begin{theorem} We have:
$$\Omega (\mathcal U) \cong {\mathcal A (3)} ^{ \Z _3}.$$
\end{theorem}

\begin{proof}
First we notice that 
\bea  && {\mathcal A (3)} ^{ \Z _3} \cong \bigoplus _{m_1  + m_2 + m_3   \in  3 {\Z
} } \pi_ {m_1}   \otimes \pi _{m_2} \otimes \pi_{m_3}   = \bigoplus _{m_1  + m_2 + m_3   \in  3 {\Z
} } \left( {\mathcal A} (1) ^{0} \right) ^{\otimes 3} .  \mathcal G_1 ^{m_1} \mathcal G_2 ^{m_2}   \mathcal G_3 ^{m_3}{\bf 1} \label{for-1}\eea
Using realization we see that
\begin{itemize}
\item  $\Omega (\mathcal U) \subset {\mathcal A (3)} ^{ \Z _3}$; 
\item $\mathcal G_1 ^{m_1} \mathcal G_2 ^{m_2}   \mathcal G_3 ^{m_3}{\bf 1}  \in \Omega (\mathcal U)$ for all $(m_1, m_2, m_3) \in {\Z} ^3$, $m_1 + m_2 + m_3 \in 3 {\Z}$.
\end{itemize}
 Now claim follows by (\ref{for-1}).
\end{proof}

\subsection {Decomposition for the general case $n \ge 3$}
For $s \in {\Z}$ we define $$ Q_n   ^{(s)} = \{  z= (z_1, \dots, z_n) \in {\Z} ^n  \ \vert \ z_1 + \cdots + z_n = s \}. $$ For $z \in Q_n ^{(s)}$ we define 
\bea 
\gamma_z & =& z_1 \varphi_1 + \cdots + z_n \varphi_n \nonumber \\
P_z & = & = X_1 ^{z_1} \cdots X_n ^{z_n} {\bf 1 } \in  W_{(n)} ^ {(0)}  \cong V_{-1} (sl(n) ) \nonumber  \\
v_z &=& \mathcal G_1 ^ {z_1} \cdots \mathcal G_n ^{z_n} {\bf 1}   \in \mathcal A (n) ^{(s)} \nonumber 
\eea
We have
$$ P_z = \nu  v_{z} \otimes e^{\gamma_z} \quad (\nu = \pm 1). $$
Set $Q_n := Q_n ^{(0)}$. 
\begin{theorem}
 \label{decomposition-sl(n)} The vertex algebra $V_{-1} (\frak{sl}(n))$ is a simple current extension of $( \mathcal A(1) ^0 ) ^{\otimes n}  \otimes M_{n-1} (1)$  and
$$ V_{-1} (\frak{sl}(n)) =  \bigoplus_{  z \in Q_n}    
\pi_{z_1}  \otimes \pi_{z_2} \cdots  \otimes \pi _{z_n} \otimes M_{n-1} (1)  e^ { \gamma_ z }.  $$ 
For $s \in {\Z}$ we have:
$$  V_s  =  \bigoplus_{  z \in Q_n^{(s)} }    
\pi_{z_1}  \otimes \pi_{z_2} \cdots  \otimes \pi _{z_n} \otimes M_{n-1} (1)  e^ { \gamma_ z }.  $$ 
\end{theorem}

\subsection {Decomposition in the  case $n = 2$}

Let us consider also the case $n=2$.
\begin{theorem}
 \label{decomposition-sl(2)} The vertex algebra $\mathcal W = \mbox{\rm Com} (M_c(1), W_{(2)} )$  (which is  isomorphic to affine $W$--algebra $W_{-5/2} (\mathfrak{sl}(4), f_{sh} ) $) is a simple current extension of $( \mathcal A(1) ^0 ) ^{\otimes 2}  \otimes M_{1} (1)$  and
$$  \mathcal W =  \bigoplus_{  m \in {\Z}  }    
\pi_{m }  \otimes \pi_{-m}   \otimes M_{1} (1)  e^ { m (\varphi_1 - \varphi_2)   }  $$ 
\end{theorem}

\begin{remark}
A detailed study of the representation theory of the vertex algebra $\mathcal W$  will be discussed in our forthcoming paper
 \cite{AMPd}.
\end{remark}


 \section{The character ${\rm ch}[\mathcal U]$ for $\frak{g}=\frak{sl}(3) $}

We now discuss graded dimensions (or characters) of ordinary $V_{-1}(\frak{sl}(3))$-modules. This was thoroughly analyzed in \cite{BMM}.

Next formula is a consequence of the explicit construction of modules (here $s \geq 0$):
\begin{equation} \label{charUs}
{\rm ch} [V_s](\tau)= q^{h_s+1/6} {\rm Coeff}_{\zeta^s} \prod_{n=1}^\infty \frac{(1-q^n)}{(1-\zeta q^{n-1})^3(1-\zeta^{-1}q^n)^3},
\end{equation}
where $h_s$ is the lowest conformal weight of $V_s$ and we also used that $c=-4$. We also have the full character formula of Kac and Wakimoto \cite{KW2001}
\begin{align*} \label{charUsfull}
& {\rm ch} [V_s](\tau; z_1,z_2)=q^{h_s+1/6} \\
& \cdot  {\rm Coeff}_{\zeta^s} \prod_{n=1}^\infty \frac{(1-q^n)}{(1-\zeta q^{n-1})(1-\zeta z_2 q^{n-1})(1-\zeta z_1 z_2  q^{n-1})(1-\zeta^{-1} z_1^{-1} z_2^{-1} q^n)(1-\zeta^{-1}z_2^{-1}q^n)(1-\zeta^{-1} q^n) }.
\end{align*} 

 
Very recently, Kac and Wakimoto \cite{KW} gave another Weyl-Kac type character  formula for ${\rm ch}[V_s]$ expressed as a rank two Jacobi false theta function (see also \cite{BMM} for a different formula).
After a specialization $(z_1, z_2) \to (1,1)$ we easily get 
\begin{align*}
F_0(q):=&(q;q)_\infty^8 q^{-1/6} {\rm ch}[V_0](\tau) \\
& = 4 \sum_{\substack{n_1\in\mathbb{N}_0\\ n_2\in\Z}} \left(2n_1-n_2+\frac12\right)\left(2n_2-n_1+\frac12\right) (n_1+n_2+1) q^{2n_1^2+2n_2^2-2n_1n_2+n_1+n_2}.
\end{align*}
We also have
\begin{align*}
& F_s(q) : =(q;q)_\infty^8 q^{-1/6-h_s} {\rm ch}[V_s](\tau) \\
& = 4 \sum_{\substack{n_1\in\mathbb{N}_0\\ n_2\in\Z}} \left(2n_1-n_2+\frac{s}{2}+\frac12\right)\left(2n_2-n_1+\frac12\right) \left(n_1+n_2+\frac{s}{2}+1\right) q^{2n_1^2+2n_2^2-2n_1n_2+(s+1)n_1+n_2},
\end{align*}
where $(q;q)_\infty = \prod_{i \geq 1} (1-q^i)$.

Observe that the summation over $n_1$ is only over the set of non-negative integers.
On the other hand, it is easy to see that the sum over the integers vanishes (by changing $n_j\mapsto -n_j-1$)
$$
 \sum_{\substack{n_1\in \mathbb{Z} \\ n_2\in\Z}} \left(2n_1-n_2+\frac12\right)\left(2n_2-n_1+\frac12\right) (n_1+n_2+1) q^{2n_1^2+2n_2^2-2n_1n_2+n_1+n_2}=0.
$$
Let
$$G(\tau) := 4 \sum_{\substack{n_1\in-\mathbb{N}_0\\ n_2\in\Z}} \left(2n_1-n_2+\frac12\right)\left(2n_2-n_1+\frac12\right) (n_1+n_2+1) q^{2n_1^2+2n_2^2-2n_1n_2+n_1+n_2}
$$
\begin{proposition}
\label{P:F-F3}
We have
$$G(\tau)=q^3 F_3(\tau).$$
\end{proposition}
\begin{proof} Straightforward computation with $q$-series.
\end{proof}

\begin{corollary} \label{s=0}
$$F_{n_2=0}(q):=F_0(q)+q^3 F_3(q),$$
where $$F_{n_2=0}(q):= 4 \sum_{\substack{n_1\in \mathbb{Z} }} \left(2n_1+\frac12\right)\left(-n_1+\frac12\right) (n_1+1) q^{2n_1^2+n_1}.$$
Moreover, this series is quasi-modular.
 \end{corollary}
\begin{proof}
Follows directly from  the previous proposition and vanishing of the sum over the full lattice. Quasi-modularity is clear as this series is obtained 
by differentiating a unary theta function.
\end{proof}

Now we combine characters of modules appearing in the decomposition of $\mathcal U$  in pairs:  
 $$q^{-1/6}(q;q)_\infty^8 {\rm ch}[\mathcal U](\tau)=\sum_{m \geq 0} (q^{ \frac{3m^2}{2}+\frac{3m}{2}} F_{3m}(q)+q^{\frac{3(m+1)^2}{2}+\frac{3(m+1)}{2}} F_{3m+3}(q)),$$
 where $(q;q)_\infty=\prod_{i=1}^\infty (1-q^i)$.
For each pair in the summation we have a similar $q$-series identity (the proof is almost identical)
\begin{lemma} \label{gen-rec} For every $m \geq 0$, 
\begin{align*}
q^{ \frac{3m^2}{2}+\frac{3m}{2}} &  F_{3m}(q)+q^{\frac{3(m+1)^2}{2}+\frac{3(m+1)}{2}}F_{3m+3}(q) \\
& = 4\sum_{n_1 \in \mathbb{Z}} (2n_1+\frac{1}{2})(-n_1+\frac{3m}{2}+\frac{1}{2})(n_1+\frac{3m}{2})
q^{2n_1^2+n_1+\frac{3}{2}m^2+\frac{3m}{2}}.
\end{align*}
\end{lemma}


 
 \begin{theorem} \label{sl3-char-schar}
 We have  
(i)
$$ {\rm ch}[\mathcal U](\tau)={\rm tr}_{\mathcal{U}} q^{L(0)-c/24}$$
$$=\frac{4}{2 \eta(\tau)^8} \sum_{n,m \in \mathbb{Z}} (2n+\frac{1}{2})(-n+\frac{3m}{2}+\frac{1}{2})(n+\frac{3m}{2}+1)q^{2(n+1/4)^2+\frac{3}{2}(m+1/2)^2}$$

(ii) Denote by $\mathcal U_{\pm 1}=\oplus_{n \in \mathbb{Z}} U_{3n \pm 1}$. Then ${\rm ch}[\mathcal U_{1}](\tau)={\rm ch}[\mathcal U_{-1}](\tau)$
and 
$$ {\rm ch}[\mathcal U_{\pm 1}](q)={\rm tr}_{\mathcal{U}_{\pm 1}} q^{L(0)-c/24}$$
$$=\frac{4}{2 \eta(\tau)^8} \sum_{n,m \in \mathbb{Z}} (2n+\frac{1}{2})(-n+\frac{3m}{2}+\frac{1}{2}+\frac{1}{2})(n+\frac{3m}{2}+\frac{1}{2}+1)q^{2(n+1/4)^2+\frac{3}{2}(m+5/6)^2}$$
(iii) For the supercharacter we have
$$ {\rm sch}[\mathcal U](q)={\rm tr}_{\mathcal{U}} \sigma q^{L(0)-c/24}$$
$$=\frac{4}{ \eta(\tau)^8} \sum_{n,m \in \mathbb{Z}} (2n+\frac{1}{2})^2 (-n+\frac{3m}{2}+\frac{1}{2})(n+\frac{3m}{2}+1)q^{2(n+1/4)^2+\frac{3}{2}(m+1/2)^2}$$
(iv) Both ${\rm ch}[\mathcal{U}]$ and ${\rm sch}[\mathcal{U}]$ are quasi-modular.
 \end{theorem}
 
 \begin{proof}
 We only prove (i) here - formula (ii) can be proven using similar ideas. Proof of (iii) is slightly different and is postponed for Section 8 (see Remark 7).  
 
 We first apply Lemma \ref{gen-rec} to write 
 $${\rm ch}[\mathcal U](\tau)={\rm tr}|_{\mathcal U} q^{L(0)+\frac{1}{6}}= \sum_{m \geq 0} \frac{q^{1/3}}{(q;q)^8_\infty} F(m)$$
 where $F(m):=q^{ \frac{3m^2}{2}+\frac{3m}{2}} F_{3m}(q)+q^{\frac{3(m+1)^2}{2}+\frac{3(m+1)}{2}} F_{3m+3}(q)$, in the form  
 $${\rm ch}[\mathcal U](\tau)=  \frac{4}{\eta(\tau)^8} \sum_{m \geq 0} \sum_{n \in \mathbb{Z}, m \geq 0} (2n+\frac{1}{2})(-n+\frac{3m}{2}+\frac{1}{2})(n+\frac{3m}{2}+1)q^{2(n+1/4)^2+\frac{3}{2}(m+1/2)^2}$$
Now observe that 
$${\rm ch}[\mathcal U](\tau)= \frac{4}{\eta(\tau)^8} \sum_{m \geq 0} \sum_{n \in \mathbb{Z}, m < 0} (2n+\frac{1}{2})(-n+\frac{3m}{2}+\frac{1}{2})(n+\frac{3m}{2}+1)q^{2(n+1/4)^2+\frac{3}{2}(m+1/2)^2},$$
so if take summation over $m \in \mathbb{Z}$ we have to divide by $2$. This implies formula (i).

In part (iv) we prove quasi-modularity only for ${\rm ch}[\mathcal{U}]$. In other cases proof is very similar.
Using 
 $$\left(2n+\frac{1}{2}\right)\left(-n+\frac{3m}{2}+\frac{1}{2}\right)\left(n+\frac{3m}{2}+1\right)=-2(n+1/4)^3+9/2(m+1/2)^2(n+1/4).$$
we get
 $$\sum_{n,m \in \mathbb{Z}} ( -2(n+1/4)^3+9/2(m+1/2)^2(n+1/4) ) q^{2(n+1/4)^2+\frac{3}{2}(m+1/2)^2}$$
 $$=- \sum_{m \in \mathbb{Z}} q^{\frac{3}{2}(m+1/2)^2} \Theta_q  \left( \sum_{n \in \mathbb{Z}} (n+1/4)q ^{2(n+1/4)^2} \right)
 + 3 \Theta_q (\sum_{m \in \mathbb{Z}} q^{\frac{3}{2}(m+1/2)^2} ) \sum_{n \in \mathbb{Z}} (n+1/4)q ^{2(n+1/4)^2}, $$
  where $\Theta_q:=q \frac{d}{dq}$.  As it is known $\Theta_q$-derivative of weight $\frac{1}{2}$ and $\frac{3}{2}$ theta functions give quasi-modular forms.

 \end{proof}
 
 \begin{remark}
Observe that irreducible $\mathcal U$-modules are also $\mathbb{Z}_2$-graded. their supercharacters are given by
  $${\rm sch}[\mathcal U_1]=-{\rm ch}[V_1] +\sum_{i \geq 1}(-1)^{i-1} ({\rm ch}[V_{3i+1}]+{\rm ch}[V_{-3i+1}]),$$
  $${\rm sch}[\mathcal U_2]={\rm ch}[V_2] +\sum_{i \geq 1}(-1)^{i-1} ({\rm ch}[V_{3i+2}]+{\rm ch}[V_{-3i+2}]).$$
Because of ${\rm ch}[V_i]={\rm ch}[V_{-i}]$ this easily implies  ${\rm sch}[\mathcal U_1]={\rm sch}[\mathcal U_2]$. One can also show 

\begin{align*}
{\rm sch}[\mathcal U_1](\tau)&={\rm sch}[\mathcal U_2](\tau) \\
&=\frac{4}{ \eta(\tau)^8} \sum_{n,m \in \mathbb{Z}} (2n+\frac{1}{2})^2 (-n+\frac{3m}{2}+\frac{1}{2})(n+\frac{3m}{2}+1)q^{2(n+1/4)^2+\frac{3}{2}(m+5/6)^2}.
\end{align*}
\end{remark} 
 
 \begin{remark}
  The above approach to modularity is difficult to generalize to $\frak{sl}(n)$ because it requires explicit formulae as in Theorem 5.4. But these formulas are non-trivial to extract 
  from \cite{KW}.  In the remaining of the paper we show how to solve the (quasi)-modularity problem
 via meromorphic Jacobi forms and explicit construction. 
  \end{remark}

 \section{Quasi-modularity of ${\rm (s)ch}[\mathcal U](\tau)$}
 
 In this part we prove the quasi-modularity of ${\rm (s)ch}[\mathcal U](\tau)$, generalizing our explicit computation for $\frak{g}=\frak{sl}(3)$ in Section 6. Let
 $$(a)_\infty:=\prod_{i \geq 1} (1-aq^{i-1}).$$
 We will make use of a Jacobi theta function $$\vartheta(z;\tau):= (-i) q^{1/8}\zeta^{-1/2} (q)_\infty (\zeta)_\infty (q \zeta^{-1})_\infty ,$$ where 
$\zeta=e^{2 \pi i z}$. Recall the elliptic and modular transformation formulae (here $\lambda,\mu\in\Z$, 
 $\left(\begin{smallmatrix}
a & b \\ c & d
\end{smallmatrix}\right)\in SL_2(\Z)$)
\begin{align*}
\vartheta(z+\lambda \tau + \mu) &= (-1)^{\lambda+\mu} q^{-\frac{\lambda^2}{2}} \zeta^{-\lambda} \vartheta(z),\\
\vartheta\left(\frac{z}{c\tau+d};\frac{a\tau+b}{c\tau+d}\right)&=\chi\left(\begin{matrix}
a & b \\ c & d
\end{matrix}\right) (c\tau+d)^{\frac12} e^{\frac{\pi i cz^2}{c\tau+d}}\vartheta(z;\tau),\notag
\end{align*}
where $\chi$ is a certain multiplier. In particular
\begin{align*}
\vartheta\left(\frac{z}{\tau};-\frac{1}{\tau}\right)&=-i\sqrt{-i\tau}e^{\frac{\pi iz^2}{\tau}}\vartheta(z;\tau).\qquad
\end{align*}

As in the $\frak{sl}(3)$ case, from the explicit construction \cite{KW2001}, 
\begin{equation} \label{formula}
\tilde{{\rm ch}}(V_{s})={\rm Coeff}_{\zeta^{s}}  \frac{ (q)_\infty }{(\zeta)^n_\infty (q\zeta^{-1})^n_\infty}
\end{equation}
where $\tilde{{\rm ch}}(V_s)$ is the character of $V_s$ up to a multiplicative $q$-shift. More precisely, for $s \geq 0$
$$\tilde{{\rm ch}}(V_s)={\rm dim}(L(s \omega_1))+ O(q)$$
and for $s<0$
$$\tilde{{\rm ch}}(V_s)=q^{-s}({\rm dim}(L(s \omega_{n-1}))+ O(q)),$$
where $L( m \omega_i)$ denotes an irreducible $\frak{sl}(3)$-module of highest weight $m \omega_i$.
Thus in order to compute the genuine character we must multiply with
$$q^{h_{V_s}-c_{n}/24},$$
for $s \geq 0$, and in addition shift with $q^{s}$ for $s<0$.
It is easy to see that for $s \geq 0$
$$h_{V_s}=h_{V_{-s}}=\frac{s^2}{2n} +\frac{s}{2}$$
and the central charge is
$$c_{n}=-(n+1).$$
Combined 
$$h_{V_{n s}}-\frac{c_{n}}{24}=\frac{s^2 n}{2}+\frac{s n}{2}+\frac{n+1}{24}.$$
Putting this together with (\ref{formula}), and taking into account the $q$-multiplicative shift for $s<0$, we get

$${\rm ch}[\mathcal U](\tau)={\rm CT}_{\zeta}  \frac{ \sum_{s \in \mathbb{Z}} q^{\frac{s^2 n}{2}+\frac{s n}{2}+\frac{n+1}{24}} \zeta^{-s n} \prod_{i=1}^\infty (1-q^i)}{(\zeta)^n_\infty (q\zeta^{-1})^n_\infty}.$$
Next we multiply the numerator and the denominator with $\zeta^{n/2} q^{n/8}(q;q)^{n}_\infty$ so that in the denominator we have a power of $\theta(z,\tau)$, a weight $\frac12$ Jacobi form of index $\frac12$.
So we obtain 
$${\rm ch}[\mathcal U](\tau)={\rm CT}_{\zeta}  \frac{ \sum_{s \in \mathbb{Z}} q^{\frac{s^2 n}{2}+\frac{s n}{2}+\frac{n+1}{24}} \zeta^{-s n-n/2}  q^{n/8} (q;q)_\infty^{n} (q;q)_\infty }{ \zeta^{-n/2} (\zeta)^n_\infty (q\zeta^{-1})^n_\infty (q;q)_\infty^{n} q^{n/8}}.$$
Continuing with the numerator
$$\sum_{s \in \mathbb{Z}} q^{\frac{s^2 n^2}{2}+\frac{s n}{2}+\frac{n+1}{24}} \zeta^{-s n-n/2}=\sum_{s \in \mathbb{Z}} q^{\frac{n}{2}(s+1/2)^2-\frac{n}{8}+\frac{n+1}{24}} \zeta^{-(s+1/2)n}.$$
Notice that $q^{n/8}$ cancels out, and an extra $q^{\frac{n+1}{24}}$ term nicely combines with $(q;q)_\infty^{n+1}$ giving $\eta(\tau)^{n+1}$.
We conclude
$${\rm ch}[\mathcal U](\tau)=i^n \eta(\tau)^{n+1} {\rm CT}_{\zeta} \frac{\sum_{s \in \mathbb{Z}} q^{\frac{n}{2}(s+1/2)^2} \zeta^{-(s+1/2)n}}{\vartheta(z,\tau)^n}=i^n \eta(\tau)^{n+1} {\rm CT}_{\zeta} \frac{\sum_{s \in \mathbb{Z}} q^{\frac{n}{2}(s+1/2)^2} \zeta^{(s+1/2)n}}{\vartheta(z,\tau)^n},$$
where we also used the Jacob triple product identity in the denominator
$$\vartheta(z,\tau)=(-i) q^{1/8}\zeta^{-1/2}(\zeta)_\infty (q \zeta^{-1})_\infty=\sum_{s \in \mathbb{Z}+\frac{1}{2}} q^{s^2/2} e^{2 \pi i s(z+1/2)}= i \sum_{s \in \mathbb{Z}} (-1)^s q^{(s+1/2)^2/2} e^{2 \pi i (s+1/2) z}.$$

\subsection{$n$ is odd} In this case we have 


$${\rm ch}[\mathcal U](\tau)= \epsilon_n \eta(\tau)^{n+1} {\rm CT}_{\zeta} \left(  \frac{\vartheta( n z+\frac{1}{2} ,n \tau)}{\vartheta(z,\tau)^{n}} \right),$$
where $\epsilon_n=-i^n$,  for the character. 

\subsection{$n$ is even}
For $n$ even we also have 
$${\rm ch}[\mathcal U](\tau)= \epsilon_n \eta(\tau)^{n+1} {\rm CT}_{\zeta} \left(  \frac{\vartheta(z n+\frac{1}{2},n \tau)}{\vartheta(z,\tau)^{n}} \right)$$

\subsection{Supercharacter}
For $n$ odd we can  also compute the supercharacter

$${\rm sch}[\mathcal U](\tau)=i \epsilon_n \eta(\tau)^{n+1} {\rm CT}_{\zeta} \left( \frac{\vartheta(z n,n \tau)}{\vartheta(z,\tau)^{n}} \right)$$




\subsection{Characters of modules}
Straightforward computation gives for $0 \leq k \leq n-1$, 
$${\rm ch}[\mathcal U_k](\tau)=\eta(\tau)^{n+1} {\rm CT}_{\zeta} \frac{\sum_{s \in \mathbb{Z}} q^{\frac{n}{2}(s+\frac{k+n}{2n})^2} \zeta^{-(s+\frac{n+k}{2n})n}}{\vartheta(z,\tau)^n}.$$
For $n$ even we can write this as 
$${\rm ch}[\mathcal U_k](\tau)= \epsilon_{n,k} \eta(\tau)^{n+1} {\rm CT}_{\zeta} \left( \frac{\vartheta((z + \frac{k}{2 n})n,n \tau)}{\vartheta(z,\tau)^{n}} \right),$$
where $\epsilon_{n,k}$ is a normalization constant as above. Similarly we compute ${\rm sch}[\mathcal U_k](\tau)$.

\subsection{Quasimodularity}
Here we prove a general theorem on quasimodularity of the (super)character of $\mathcal{U}$, which extends our previous calcultions for $\frak{sl}(3)$. 

\begin{theorem} The supercharacter  of $\mathcal U$ (for $n$ odd) and the character of $\mathcal U$ (for $n$ even) are quasi-modular forms (with multipliers) of  weight one and depth one on $\Gamma_0(n)$.

\end{theorem}
\begin{proof} 

{\rm Case 1.} $n$ is odd.

First we observe that 
$$G(\tau,z):=\eta(\tau)^{n+1} \left( \frac{\vartheta(z n,n \tau)}{\vartheta(z,\tau)^{n}} \right)$$
is a meromorphic Jacobi form of weight $\frac{n+1}{2}+\frac{1}{2}-\frac{n}{2}=1$. 
After we multiply with $\frac{1}{\eta(\tau)^2}$, $H(\tau;z):=\frac{G(\tau;z)}{\eta(\tau)^2}$ weget  a meromorphic Jacobi function of weight zero. \\
\noindent {\em Claim: } $G(\tau,z)$ is a Jacobi form of index zero for the congruence group $\Gamma_0(n)$ (transforming with a character).

We consider $\left[\begin{array}{cc} a & b \\  c  & d \end{array}\right] \in \Gamma_0(n)$ (thus $ad-c b=1$ and $n | c$). Then
$$\vartheta(n z; n \tau)|_{(\tau \to \frac{a \tau+b}{c \tau+d}; z \to \frac{z}{c\tau+d})}=\vartheta \left(\frac{n z}{ c \tau + d};  n  \frac{a \tau+ b }{c  \tau+d}\right)=
\vartheta \left(\frac{(n z)}{\frac{c}{n}(n \tau)+d};\frac{a (n \tau)+n b}{\frac{c}{n}(\tau n)+d}\right)$$
$$=\chi' \left(\begin{matrix}
a & b \\ c & d
\end{matrix}\right)(c \tau+d)^{\frac12} e^{\frac{ \pi i c n z^2}{c \tau+d}}\vartheta(n z; n \tau)$$
where we used that $\left[\begin{array}{cc} a & bn \\  \frac{c}{n}  & d \end{array}\right] \in \Gamma(1)$.
For the denominator,
$$\vartheta( z;  \tau)^n |_{(\tau \to \frac{a \tau+b}{c \tau+d}; z \to \frac{z}{c\tau+d})}=\chi^n (c \tau+d)^{\frac{n}{2}}
e^{\frac{ \pi i c n z^2}{c \tau+d}}\vartheta(z; \tau)^n. $$
For translations, for $\lambda,\nu \in \mathbb{Z}$, we have 
$$\vartheta(nz+n\lambda \tau+n \mu, n \tau)=(-1)^{n(\lambda+\mu)} q^{-n \lambda^2/2} e^{-2 \pi i \lambda n z} \vartheta(nz,n \tau),$$
$$\vartheta(n+\lambda \tau+\mu,\tau)^n=(-1)^{n(\lambda+\mu)} q^{-n \lambda^2/2}e^{-2 \pi i \lambda n z} \vartheta(z, \tau)^n.$$
After taking quotient this implies the claim.

Notice that $H(\tau;z)$ is even with respect to $z$ has a pole of order $n$ at $z=0$, so we can write 
Laurent expansion \cite{BFM} (see also \cite{DMZ})

\begin{equation} \label{Laurent}
H(\tau;z)=\frac{H_n(\tau)}{(2 \pi i z)^n}+\frac{H_{n-2}(\tau)}{(2 \pi i z)^{n-2}}+\cdots + \frac{H_2(\tau)}{(2 \pi i z)^2}+H_0(\tau),
\end{equation}
where $H_{2j}(\tau)$ is a modular form of weight $-2j$ with respect to $\Gamma_0(n)$ (transforming with the same character as the Jacobi form).

Then by using \cite{BFM,DMZ} we can write the "finite" part  as 
\begin{equation} \label{dmz}
H^F(\tau):=H_0(\tau)+\sum_{j=1}^{\frac{n}{2}} \frac{B_{2j}}{(2j)!} H_{2j}(\tau)E_{2j}(\tau)
\end{equation}
which is quasi-modular of weight zero. Here $E_{2j}(\tau)$ denotes Eisenstein series and $B_{2j}$ are Bernoulli numbers. Finally, the constant term is 
$${\rm sch}[\mathcal U](\tau)=\eta(\tau)^2 H^F(\tau)$$
is of weight one.  The depth is one due to appearance of $E_2(\tau)$.

{\rm Case 2.} $n$ is even.

For $n$ is even we have to study $$H(z;\tau):=\eta(\tau)^{n-1}  \frac{\vartheta(z n+\frac{1}{2},n \tau)}{\vartheta(z,\tau)^{n}} .$$
Since 
$$\vartheta((-z)n+\frac{1}{2};\tau)=\vartheta(-zn+\frac12;\tau)=-\vartheta(zn-\frac12;\tau)=\vartheta(zn+\frac12;\tau)$$
and $n$ is even, $H(z;\tau)$ is an even function. It is easy to see that 
$$\vartheta(z+\frac{1}{2}+ \lambda \tau+\mu)=(-1)^{\mu} q^{-\lambda^2/2} e^{-2 \pi i \lambda z} \vartheta(z+\frac12;\tau)$$
which implies
$$\vartheta(nz+n\lambda \tau+n \mu+\frac{1}{2}, n \tau)= q^{-n \lambda^2/2} e^{-2 \pi i \lambda n z} \vartheta(nz,n \tau).$$
Thus $H\left(z+\lambda \tau+\mu;\tau\right)=H(z; \tau)$. We also get
$$H\left(\frac{z}{c \tau+d};\frac{a \tau+b}{c \tau+d}\right)=\chi''\left(\begin{matrix}
a & b \\ c & d
\end{matrix}\right)H(z;\tau),$$
where $\chi''$ is a character. The rest follows as in the odd case. 
 \end{proof}
 
 
 \subsection{Explicit example: case $n=3$}
 
 Here we compute the constant term of  
 
 $$G(\tau,z):=\eta(\tau)^{4} \left( \frac{\vartheta(3 z; 3 \tau)}{\vartheta(z,\tau)^{3}} \right).$$
 
 The same method can be used to compute $G(\tau;z)$ for every $n$. We have to compute modular forms appearing inside the series (\ref{Laurent}).
 Here we use a standard method of Laurent expansion following \cite{BFM}.
 We write 
 $$\vartheta^*(z;\tau):=\frac{1}{z} \vartheta(z;\tau),$$
 where we suppress $\tau$ from the formula for brevity.
 Then we have 
 $$\vartheta^*(z;\tau)=\vartheta^*(0;\tau)+{\vartheta^*}''(0;\tau)\frac{z^2}{2!} + O(z^4)$$
 and 
 $$\vartheta^*(3z;3 \tau)=\vartheta^*(0; 3 \tau)+{\vartheta^*}''(0)\frac{9z^2}{2!} + O(z^4)$$
\begin{equation} \label{laurent-2}
 G(z)=\frac{3z\left(\vartheta^*(0;3 \tau)+{\vartheta^*}''(0; 3 \tau) \frac{9 z^2}{2!} + O(z^4)\right)}{z^3(z\left(\vartheta^*(0)+{\vartheta^*}''(0)+\frac{z^2}{2!} + O(z^4)\right)} 
 \end{equation}
 It is clear that 
 $$\vartheta^*(0;\tau)=-2 \pi \eta^3(\tau)= -2 \pi \sum_{n \in \mathbb{Z}} (-1)^n (n+1/2) q^{\frac{1}{2}(n+1/2)^2}$$
 from the infinite expansion of $\vartheta(z;\tau)$ and Euler's theorem
 and 
 $${\vartheta^*}''(0;\tau)=\frac{1}{12}(2 \pi i)^2 E_2(\tau) \eta^3(\tau).$$
 Expanding (\ref{laurent-2}) gives only even powers of $z$ and in particular
 $$H_0(\tau)+\frac{H_{-2}(\tau)}{(2 \pi i z)^2} +O(1).$$
 Finally 
 \begin{align*}
&  {\rm CT}_{z} \left\{ \eta(\tau)^{4} \left( \frac{\vartheta(3 z; 3 \tau)}{\vartheta(z,\tau)^{3}} \right) \right\}=H_0(\tau)+\frac{B_2}{2} H_2(\tau)E_2(\tau) \\
& = -\frac{9}{8} E_2(\tau) \frac{\eta(\tau)^3 \eta(3 \tau)^3}{\eta(\tau)^8}+\frac{9}{8} \frac{\eta(\tau)^3 \sum_{n \geq 0}(-1)^n (2n+1)^3 q^{3 n(n+1)/2}}{\eta^8(\tau)}+
\frac{\eta^3( 3 \tau) \sum_{n \geq 0} (-1)^n (2n+1)^3 q^{ n(n+1)/2}}{\eta^8(\tau)} \\
&=-\frac{1}{8} E_2(\tau) \frac{\eta(\tau)^3 \eta(3 \tau)^3}{\eta(\tau)^8}+ \frac{9}{8} \frac{\eta(\tau)^3 \sum_{n \geq 0}(-1)^n (2n+1)^3 q^{3 n(n+1)/2}}{\eta^8(\tau)} \\
&=-\frac{1}{8} \left( E_2(\tau) - 9 E_2(3 \tau) \right) \frac{\eta(3 \tau)^3}{\eta(\tau)^5.}
\end{align*}
\begin{remark} \label{super-proof}
Observe that the above formulas, with $\eta(\tau)^3$ and $\eta(3 \tau)^3$ expanded as sums, immediately imply relation Theorem \ref{sl3-char-schar}, (iii).
\end{remark}

 It is clear that  $E_2(\tau)$ is a quasi-modular form of weight $2$ and depth 1 on $\Gamma(1)$. It is easy to show that 
 $E_{2,3}(\tau):=E_2(3 \tau)$ is a quasi-modular form of weight $2$ and depth $1$ on $\Gamma_0(3)$, i.e. 
 $$E_{2,3}\left(\frac{a \tau+b}{c \tau+d}\right)=(c \tau+d)^2 E_{2,3}(\tau)+\frac{6 c(c \tau+d)}{3 i \pi}.$$
As index of $\Gamma_0(3)$ in $\Gamma(1)$ is $4$, $E_{2,3}(\tau)$ combines into a vector-valued quasi-modular form under the full modular group.
\begin{lemma} \label{lemma-eta3} $\eta(\tau/3)^3$ and $\eta(3 \tau)^3$ form a 2-dimensional vector-valued modular form of weight $\frac{3}{2}$ under the full modular group.
\end{lemma}
\begin{proof}
Straightforward computation with Shimura's theta series of weight $\frac{3}{2}$ together with Jacobi's identity for $\eta(\tau)^3$.
\end{proof}
Using this lemma and previous discussion one can explicitly write down a vector space of quasi-modular forms closed under the modular group, which also contains the supercharacter. However, this space is difficult to analyze and does not give much evidence for the quasi-lisseness  of $\mathcal U$ conjectured earlier. As demonstrated in \cite{AK16}, characters of quasi-lisse vertex algebras must satisfy a particular type 
of linear modular differential equation whose coefficients are  {\em holomorphic} Eisenstein series (usually abbreviated as MLDE). For  quasi-lisse $\mathbb{Z}_{\geq 0}$-graded vertex superalgebras 
we expect the same property to hold for supercharacters. By analyzing the leading behavior of the above function (which is quasi-modular) combined with computer computations we can conclude
\begin{proposition} \label{MLDE} ${\rm sch}[\mathcal U](\tau)$ satisfies a $5$-th order MLDE 
$$\theta^5 (y(q)) -\frac{7}{36} E_4(\tau) \theta^3 (y(q))+\frac{19}{216} E_6(\tau) \theta^2 (y(q))-\frac{5}{324}E_4(\tau)^2 \theta(y(q))+\frac{5}{1944} E_4(\tau)E_6(\tau) y(q)=0,$$
where Ramanujan-Serre's $n$-th derivative is defined by 
$$\theta^n:=\vartheta_{2n} \circ \cdots \circ \theta_{0}  ; \ \  \theta_k:=\left(q\frac{d}{dq}-\frac{k E_2(\tau)}{12}\right).$$
\end{proposition}
As usual, the Eisenstein series in the equation are given by 
\begin{align*}
E_2(\tau)&=1-24 \sum_{n \geq 1} \frac{nq^n}{1-q^n} \\
E_4(\tau)&=1+240 \sum_{n \geq 1} \frac{n^3q^n}{1-q^n} \\
E_6(\tau)&=1-504 \sum_{n \geq 1} \frac{n^5q^n}{1-q^n}.
\end{align*}

Two supercharacters (that are equal) of ordinary $\mathcal{U}$-modules also satisfy this modular equation. Two additional solutions are expected to come from $\sigma$-twisted $\mathcal{U}$-modules, not analyzed in this paper. These four solutions together with a logarithmic solution form a fundamental system of this MLDE.

\subsection{Explicit example: case $n=2$}
Here we essentially repeat the same procedure with a notable 
difference that 
$$\vartheta\left(2 z+\frac{1}{2}; 2 \tau\right)$$
admits Taylor expansion in even powers of the $z$ variable:
$$\vartheta \left(2 z+\frac{1}{2}; 2 \tau\right)=\vartheta \left(\frac{1}{2},2 \tau \right)+ O(z^2)$$
so that  
$$\frac{\vartheta(2 z+\frac{1}{2}; 2 \tau)}{\vartheta(z;\tau)^2}=\frac{1}{z^2} a_0(\tau)+a_1(\tau)+O(z^2)$$
Repeating the same procedure as in the odd supercharacter case we get
$${\rm ch}[\mathcal U](\tau)=\frac{1}{3} \frac{\eta(4 \tau)^2}{\eta(\tau)^3 \eta(2 \tau)} \left( 4 E_2(2 \tau) - E_2(4 \tau) \right).$$

 \section{ Vertex superalgebra $V_1 (\frak{psl}(n,n) ) $ and $V_{-1}(\frak{sl}_n)$} 
 
 Let $\g = \frak{psl}(n,n)$. We consider the simple vertex algebra $V_1(\g)$.
 We have the following result which identifies our vertex algebra $\mathcal U ^{(n)} = \mathcal U_0$ as a coset subalgebra in $V_1(\g)$.
 \begin{proposition} Assume than $n \ge 3$. Then we have:
 \item [(1)] The vertex algebras $\mathcal U^{(n)}$ and $L_{\frak{sl}_n}( \Lambda_0)$ form a Howe dual pair inside  $V_1 (\g)$. In particular,
 $$ \frac{\frak{psl}(n,n) _1}{\frak{sl}(n)_1}:= \mbox{Com}_{V_1(\g)} (L_{\frak{sl}_n} (\Lambda_0)) \cong \mathcal U^{(n)}. $$
 \item [(2)]  $K(\g, 1) \cong K(\frak{sl}(n), -1) \cong (\mathcal A(1) ^{0}  ) ^{\otimes n}$. 
 \end{proposition}
 \begin{proof}
By using the decomposition of conformal embedding $\frak{sl}(n) \times \frak{sl}(n) \hookrightarrow \g$  (cf. \cite{AKMPP-new}) we get 
  \bea V_1( \g) = \bigoplus_{i=0} ^{n-1}  \mathcal U_i \otimes L(\Lambda_i),   \label{dec-ce} \eea
  where for brevity we omit the superscript $\frak{sl}_n$.
  Alternatively, relation (\ref{dec-ce}) can be directly proved by using the  fusion rules result   from Theorem \ref{extended-algebra} and the well-known  fact that all  $V_1 (\frak{sl}(n))$--modules are simple currents.
  
  The first assertion follows directly from (\ref{dec-ce}).  The second assertion follows again from (\ref{dec-ce}) and from
  $$ K(\frak{sl}(n), 1) \cong {\C}, \quad \mbox{Com}_{\mathcal U^{(n)}} (M_{n-1}(1))   = K(\frak{sl}(n), -1). $$
 \end{proof}
 
  The case $n=1$ corresponds exactly to the  symplectic fermion vertex algebra $\mathcal A(1)$ of central charge $c=-2$. 
  We will see that for  $n \ge 2$, the supertrace  $\mbox{sch} [V_1 (\g)](\tau)$ are the same, and therefore they satisfy the same MLDE
  \begin{equation} \label{aa}
  \theta^2( y(\tau))+\frac{1}{144} E_4(\tau) y(\tau)=0.
  \end{equation}
 \begin{theorem} \label{psl3} We have:
$$\mbox{\rm sch} [V_1 (\g)] (\tau)  =  \eta(\tau) ^2.$$
 \end{theorem}
 
 \subsection {Proof of Theorem \ref{psl3} }
The proof of Theorem \ref{psl3} uses explicit realization of $V_1 (\frak{gl}(n \vert n))$--modules and relation between supercharacters $V_k(\frak{psl}(n\vert n))$ and $V_k (\frak{gl}(n \vert n))$.

\begin{lemma} For every $k$ we have:


$$ {\rm sch} [V_k (\frak{gl}(n, n))](\tau) =  \frac{ {\rm sch} [V_k (\frak{sl}(n, n))](\tau)] }{\eta(\tau)} =  \frac{ {\rm sch} [V_k (\frak{psl}(n, n))](\tau)] }{\eta(\tau) ^2}. $$

\color{black}

\end{lemma}
\begin{proof}
Follows from  the definition of $\frak{psl}(n,n)=\frak{gl}(n,n) / I$, where $I$ is a two-dimensional abelian ideal.
\end{proof}
\begin{lemma}  For $k=1$ we have
$$ {\rm sch} [V_k (\frak{gl}(n, n))](\tau) =   1.  $$ 
\end{lemma}
\begin{proof}
Recall that $V_1(\frak{gl}(n,n))$ is realized as a charge-zero component of  the vertex algebra $W_{(n)} \otimes F_{(n)}$, where $W^{(n)}$ is the Weyl vertex algebra, $F^{(n)}$ is the Clifford vertex algebra. and charge operator is $J(0)$ where
$$ J= \sum_{i=1} ^n  \left( : a^+ _i a^- _i: + :\Psi_i ^+ \Psi^- _i : \right). $$
Since
\bea   {\rm sch} [ W^{(n)} \otimes F^{(n)} ](\tau) &=&   {\rm sch}  [ F^{(n)} ](\tau)   \cdot   {\rm ch}  [ W^{(n)} ](\tau) \nonumber \\ & =& \frac{ \left( \prod_{m=1} ^{\infty} ( 1 - q^{m-1/2} z) (   ( 1 - q^{m-1/2} z^{-1} )  \right) ^{n}} {  \left( \prod_{m=1} ^{\infty} ( 1 - q^{m-1/2} z) (   ( 1 - q^{m-1/2} z^{-1} )  \right) ^{n} } \nonumber  \\ &=& 1 \nonumber \eea
we conclude that  ${\rm sch} [V_1 (gl(n, n))](\tau) =   1$. 
(${\rm sch} [V_1 (\frak{gl}(n, n))](\tau) $ is exactly the  constant term of the expression above.  The coefficient  of $z^j$, $j \ne 0$,  gives the supercharacters of  an  irreducible  $V_1 (\frak{gl}(n, n))$--module, and it is zero.)
\end{proof}

 Now Theorem \ref{psl3} follows from previous two lemmas.
 
 \begin{remark}
 Theorem \ref{psl3} is also in agreement with the recent results on the Duflo-Serganova functor \cite{GS}.
 \end{remark}
 
 \subsection{Second proof of Theorem \ref{psl3} for $n=3$.}
 In the case $n=3$, we have a different proof which uses  the branching rules for conformal embeddings.
 
 We have 
  $$ V_1( \g) =\mathcal U_0 \otimes L(\Lambda_0) \bigoplus \mathcal U_1 \otimes L(\Lambda_1) \bigoplus \mathcal U_2 \otimes L(\Lambda_2).   $$
  This gives
 $$  {\rm sch}[V_1( \g)]( \tau)   =  \sum_{i=0} ^{n-1} {\rm sch} [\mathcal U_i] (\tau) {\rm ch} [L(\Lambda_i)](\tau). $$
 Since both left and right hand side are (quasi)modular in theory it would be sufficient to compute a few first coefficients in the $q$-expansion.
Here we present a more conceptual proof.
We need an auxiliary result 
\begin{lemma}
 \begin{align*}
 {\rm ch}[\Lambda_0](\tau)&=\frac{\sum_{m,n \in \mathbb{Z}} q^{m^2+n^2-mn}}{\eta(\tau)^2}=\frac{1}{\eta(\tau)^3} \left(3 \eta(3 \tau)^3+\eta(\tau/3)^3 \right) \\
 {\rm ch} [\Lambda_i](\tau)&=\frac{\sum_{m,n \in \mathbb{Z}} q^{m^2+n^2+n+\frac{1}{3}-mn}}{\eta(\tau)^2}=\frac{3 \eta(3 \tau)^3}{\eta(\tau)^3},
 \end{align*}
\end{lemma}
\begin{proof}
The second identity is essentially the Macdonald denominator identity for $A_2$. By Lemma \ref{lemma-eta3} we have 
$$\mathcal V:={\rm Span} \left\{ \frac{3 \eta(3 \tau)^3}{\eta(\tau)^3}, \frac{3 \eta( \tau/3)^3}{\eta(\tau)^3} \right\}$$
is modular invariant. On the other hand
$$W:={\rm Span} \{ {\rm ch}[\Lambda_0](\tau), {\rm ch} [\Lambda_1](\tau) \}$$ 
is also two-dimensional modular invariant subspace. Since $ {\rm ch}[\Lambda_1](\tau) \in \mathcal V$ we must have $ {\rm ch}[\Lambda_0](\tau) \in \mathcal V$.
This quickly gives the formula by comparing the leading coefficients in the $q$-expansion. 
\end{proof}

 \begin{proposition}
 The Thereom \ref{psl3} holds for $n=3$.  
 \end{proposition}
 \begin{proof}
The above lemma gives
 \begin{align*}
 {\rm ch}[\Lambda_0](\tau)&=\frac{1}{\eta(\tau)^3} \left(3 \eta(3 \tau)^3+\eta(\tau/3)^3 \right) \\
 {\rm ch} [\Lambda_i](\tau)&=\frac{3 \eta(3 \tau)^3}{\eta(\tau)^3}.
 \end{align*}
 
 As in the previous section, for $1 \leq i \leq 2$ we get 
\begin{align*}
{\rm sch}[\mathcal U_i](\tau)&=- \sum_{n \geq 0; n \equiv \pm 1 \mod 6} {\rm ch}[V_n]+\sum_{n \geq 0; n \equiv \pm 2 \mod 6} {\rm ch}[V_n] \\
&= \frac{1}{6} \frac{\eta(\tau)^5}{\eta(3 \tau)^3}+\frac{(E_2(\tau)-9 E_2(3 \tau))(\eta(\tau/3)^3+3 \eta(3\tau)^3)}{48 \eta(\tau)^5}.
\end{align*}
We previously derived the formula
$${\rm sch}[\mathcal U_0](\tau)=-\frac{1}{8} \left( E_2(\tau) - 9 E_2(3 \tau) \right) \frac{\eta(3 \tau)^3}{\eta(\tau)^5}.$$
Plugging-in these $q$-series gives 
 $$  \sum_{i=0} ^{2} {\rm sch} [\mathcal U_i] (\tau) {\rm ch} [L(\Lambda_i)](\tau) = \eta(\tau)^2$$
 as desired.
 \end{proof}
 

 \subsection{MLDE for the character of $\frak{psl}(n|n)_1$}
 
We have
 \begin{conjecture} For every $n \geq 2$,  the character of $\frak{psl}(n|n)_1$ satisfies the following second order MLDE (of weight zero):
 \begin{equation}\label{MMLDE-pslnn}
\left(q \frac{d}{dq} \right)^2 y(\tau)-\frac{1}{6} E_2(\tau)\left(q \frac{d}{dq}\right) y(\tau)+\left(-\frac{6n^2-5}{720} E_{4}(\tau)+\frac{n^2}{120}E_{4,2}(\tau)\right)
 y(\tau)=0,
 \end{equation}
 \end{conjecture}
 where $$E_{4,2}(\tau)=1-240 \sum_{m \geq 1} \frac{m^3 q^m}{1+q^m}$$
 is an Eisenstein series on $\Gamma_0(2)$.
 We are able to prove this in a few low rank cases.
\begin{proposition} The conjecture is true for $2 \leq n \leq 4$.
\end{proposition}
\begin{proof}
For $n$ even we only comment on $n=2$ as $n=4$ is very similar. 
In this former case the character is \cite{BFM} 
$$y(q):={\rm ch}[\frak{psl}(2|2)_1](\tau)=\frac{\eta(2 \tau)^4}{\eta(\tau)^6} \left(\frac13 E_2(\tau)-\frac43 E_2(2 \tau) \right).$$
As the logarithmic derivative of the $\eta$-quotient contributes only with a linear combination of $E_2(\tau)$ and $E_2(2 \tau)$, plugging in $y(q)$ inside the left-hand side of of the MLDE leaves as with the same $\eta$-quotient multiplied with a  quasi-modular form of weight $6$.
As we know this ring is generated by $E_{2,2}(\tau)$, $E_{4}(\tau)$ and $E_2(\tau)$, so in order to prove that $y(q)$ satisfies (\ref{MMLDE-pslnn}) we only have to compute 
the first three coefficients in the $q$-expansion and show that they are zero (as there cannot be such a form of weight $6$ with the order of vanishing at $i \infty$ bigger than $3$). This can be easily inspected with a computer.

For $n=3$, the character is modular \cite{BFM} and computation as before gives 
$$y(q):={\rm ch}[\frak{psl}(3|3)_1](\tau)=\frac{\eta(2\tau)^{6}}{\eta(\tau)^{12}} E_{2,2}(\tau),$$
where $E_{2,2}(\tau)=1+24 \sum_{n \geq 1} \frac{n q^n}{1+q^n}$ is modular form of weight 2 on $\Gamma_0(2)$ with a character. Plugging in into (\ref{MMLDE-pslnn}) and applying the same argument gives the claim.
\end{proof}

 The conjecture is also true for $n=1$ (the case of symplectic fermions) 
 with solution $y(\tau)=q^{1/12} \prod_{n \geq 1} (1+q^n)^2$. Degenerate case $n=0$ gives MLDE for $\eta(\tau)^2$ discussed earlier; see (\ref{aa}). Similar 2nd order MLDE can be written 
 for the vacuum $\frak{gl}(n|n)_1$ character.
 We note that the "constant" coefficient in our MLDE can be rewritten as 
 $$-4 n^2 F_4(\tau)+\frac{1}{144}E_4(\tau)$$
 where $F_4(\tau)=q + 8 q^2 +28 q^3 + 64 q^4+...$ is the unique cusp form of weight $4$ on $\Gamma_0(2)$.
 An interesting feature of this family of MLDEs is that for {\em every} $n$ there is a unique vacuum solution of the form $q^a(1+O(q))$, where  $a$ must be $\frac{1}{12}$.
 The other (linearly independent) solution is logarithmic and can be expressed in integral form though.

 The method in cannot be used for all $n$. Instead,  we propose to prove this conjecture by emulating approach in  \cite{KK} based on recursions of the family of solutions.

 \color{black}
 \section{On super-characters of $V_{-2} (\frak{osp}(n+8 \vert n ))$}
 
 In \cite{AK16}, T. Arakawa and K. Kawasetsu proved  the character formula for the vertex operator algebras associated with the Deligne exceptional series 
at level $k = - \frac{h^{\vee}}{6} -1$. In \cite{AKMPP-JA} and \cite{AKMPP-2018}, the authors discovered a family of Lie superalgebras such that the  associated vertex algebras also have level $k = - \frac{h^{\vee}}{6} -1$ and similar properties as in the case of the Deligne exceptional series. Vertex algebras $V_1 (\frak{psl}(n,n) ) =  V_{-1} (\frak{psl}(n,n) )$ belong to the series. Since we  have demonstrated in previous section that the supercharacters of $V_1 (\frak{psl}(n,n) )$  should not depend on the parameter $n$, one can ask if the similar situation can be happened  in other cases. A natural example is  $V_{-2} (\frak{osp}(n+8 \vert n )$ which is a super-generalization of  the affine vertex algebra $V_{-2}(\frak{so}(8))$. We have the following conjecture (which is also in agreement with \cite{GS}):

\begin{conjecture}
For every even $n \ge 0$, we have
$${\rm sch}[V_{-2} (\frak{osp}(n+8 \vert n)](\tau)   =   {\rm ch} [V_{-2} (\frak{so}(8))](\tau) =\frac{(q \frac{d}{dq}) E_4   (\tau)}{240 \eta(\tau) ^{10}}. $$
\end{conjecture}

We plan to discuss a  proof in our forthcoming papers.

 
 
 

 \end{document}